\documentclass[12pt]{amsart}
\usepackage{amsmath,amsfonts,amssymb,amsthm,amscd, latexsym, euscript, xypic, verbatim}
\usepackage[all]{xy}
\makeatletter
\def\section{\@startsection{section}{1}%
  \z@{1.1\linespacing\@plus\linespacing}{.8\linespacing}%
  {\normalfont\Large\scshape\centering}}
\makeatother

\theoremstyle{plain}

\newtheorem*{conj*}{Root Groups Conjecture}
\newtheorem*{thm1.2}{(1.2) Theorem}
\newtheorem*{thm1.3}{(1.3) Theorem}
\newtheorem*{thm1.4}{(1.4) Theorem}
\newtheorem*{prop*}{Proposition}

\newtheorem{prop}{Proposition}[section]

\newtheorem{thm}[prop]{Theorem}
\newtheorem{cor}[prop]{Corollary}
\newtheorem{lemma}[prop]{Lemma}

\theoremstyle{definition}
\newtheorem{Def}[prop]{Definition}
\newtheorem*{Def*}{Definition}
\newtheorem{Defs}[prop]{Definitions}

\newtheorem{examples}[prop]{Examples}

\newtheorem*{notation*}{Notation}
\newtheorem{remark}[prop]{Remark}

\newcommand{\zz}{\mathbb{Z}}

\newcommand{\ga}{\alpha}
\newcommand{\gb}{\beta}
\newcommand{\gc}{\gamma}
\newcommand{\gC}{\Gamma}

\newcommand{\gD}{\Delta}

\newcommand{\gr}{\rho}

\newcommand{\gt}{\tau}

\newcommand{\lan}{\langle}
\newcommand{\ran}{\rangle}

\newcommand{\onto}{\twoheadrightarrow}

\numberwithin{equation}{section}

\begin{document}
\title[]{Relative Schur multipliers and  universal extensions of group homomorphisms}
\author[Emmanuel D.~Farjoun,\quad Yoav Segev]
{Emmanuel~D.~Farjoun\qquad Yoav Segev}

\address{Emmanuel D.~Farjoun\\
        Department of Mathematics\\
        Hebrew University of Jerusalem, Givat Ram\\
        Jerusalem 91904\\
        Israel}
\email{farjoun@math.huji.ac.il}

\address{Yoav Segev\\
        Department of Mathematics\\
        Ben Gurion University\\
        Beer Sheva 84105\\
        Israel}
\email{yoavs@math.bgu.ac.il}
\begin{abstract}
In this note, starting with
 any group homomorphism $f\colon\gC\to G$, which is surjective upon abelianization,  we construct  a universal
central extension $u\colon U\onto G,$ {\em under $\Gamma$} with the same surjective property, such that for any central extension
$m\colon M\onto G,$ under $f,$ there is a unique
homomorphism $U\to M$ with the obvious commutation condition.  The kernel
of $u$ is the relative Schur multiplier group $H_2(G,\gC;\zz)$ as below.
The case where $G$ is perfect corresponds to $\gC=1$. This yields homological obstructions to
lifting solution of equations in $G.$  Upon repetition, for finite groups,  this gives a universal
hypercentral factorization of the map $f\colon\gC\to G$.
\end{abstract}
\date{\today}
\maketitle

\section{Introduction}

Let $G$ be a group.  The purpose of this note is to extend the well known
theorem saying  that if   $G$  is perfect, then there is  a universal
perfect central extension $\widehat{G}\to G,$ whose  kernel is
the Schur multiplier $H_2(G;\zz)$ (\cite{Ka}). We replace the condition  of  $G$ being perfect
with the hypothesis that   a map $\Gamma \to G$  induces a surjection on the abelian quotients.

Throughout this note we fix a group homomorphism $f\colon\gC\to G,$
and we
consider  extensions
\[
0\to K\to M\to G\to 1.
\]
%
in which often $K=A$ is  an abelian group.
The following is the basic concept used in this note (compare with the end of section 4, p.~263 of Hochschild's paper \cite{H}.)

\begin{Defs}\label{def f-central}
{\em
\begin{enumerate}
\item
An $f$-extension of $G$  is a pair $(M,\psi),$ where $M$ is an extension of $G$
and  $\psi\colon\gC\to M$ is a map--{\em  the structure map}--that factorizes  $f$
as in diagram \eqref{diag f-ext} below. If the extension is central we refer to it as
a {\em central $f$-extension.}

\begin{equation}\label{diag f-ext}
\xymatrix{
&&&\Gamma\ar[dl]_{\psi}\ar[dr]^f\\
0\ar[r]&K\ar[r]&M\ar[rr]^n&&G\ar[r]& 0\\
}
\end{equation}
\item
A map between two $f$- extensions $(M,\psi)$ and $(M',\psi')$ of $G$
with kernels $K, K'$ respectively, is a map of the underlying extensions
which is the identity on $G,$ as in the commutative diagram \eqref{diag map f-ext}
below.
\end{enumerate}
}
\end{Defs}
\smallskip

\noindent
{\bf Notation and Remarks}. Hochschild in his  relative homological algebra  paper deals with a similar situation for
subgroups $\Gamma\subseteq G$, see \cite[section 3 and 4]{H}.
Our treatment here is for general group maps and  this allows a simpler formulation.
When there are no coefficients in sight, we take (co)homology  with integral coefficients.
Relative (co)homology  groups
are calculated with respect to  our given fixed map $f$.
That is, for an abelian group $A$, by $H_*(G,\gC; A)$ we mean
 $H_*(BG\cup_{B f}{\rm Cone}(B\gC); A).$  A more careful-but still straightforward-definition of relative (co)-homology is needed for
 relative (co-)homology with local coefficients in a  $G$-module $A$. The former is used to formulate the following classification statement which is
a direct analogue of the classical absolute case \cite{H}
(see subsections 2.4, 2.5 below). We formulate it for central extensions but a similar result exists for a more general situation:
\begin{prop}\label{prop f-central}
The   equivalence classes of central $f$-extensions of $G,$ with a given  abelian group $A$ as a  kernel,
have a natural abelian group structure (see Definition \ref{def ext}) and as such are classified by
the relative cohomology group $H^2(G,\gC; A)$, with   coefficients $A$.
\end{prop}


The following is an extension of the Schur universal extension to a relative case:

\begin{thm}\label{thm f-central}
Assume that the map $f_{ab}\colon \gC_{ab}\to G_{ab}$ induced on the abelianizations is
surjective.  Then
there exists a universal central  $f$-extension
$(U,\eta)$ of $G$ with kernel  $H_2(G,\gC),$ such that
for any central $f$-extension  $(E,\psi)$ of $G,$ there
is a unique map of central $f$-extensions (as in Definition \ref{def f-central}(2))
from $U$ to $M$. Moreover the map $\eta_{ab}: \Gamma_{ab}\to U_{ab}$ is again surjective.
\end{thm}

\noindent
{\bf 1.4.\ The universal hyper-central extension}. In the absolute Schur case the universal
central extension does not have non-split extensions, and its second homology group vanishes. In the present case the relative
universal central extension has, in general,  further non-split central $f$-extensions, and  among them  a universal one.
However for finite groups this process of taking universal central extensions stops after a finite number of steps. In this way one gets a {\em universal hypercentral extension}
 with no further non-split $f$-central ones. By hyper central extension one simply means a  composition $E_n\to E_{n-1}\to \cdots\to E_1\to E_0=G$
 of a finite number of  central extensions $E_{i+1}\twoheadrightarrow E_i.$

Thus one gets a tower of universal  central $f$-extensions:
\[
\dots\ \twoheadrightarrow U_n\twoheadrightarrow U_{n-1}\twoheadrightarrow\cdots U_1=U\twoheadrightarrow G.
\]
 As mentioned above, for a map of finite groups this tower terminates with $U_\infty=U_N$ a finite group yielding
a factorization $\Gamma\to U_\infty\to G.$ This
$U_\infty$ is the universal, both initial and terminal,  factorization of the map $f$  in the following sense:

\setcounter{prop}{4}
\begin{thm}\label{thm uni U}
Let  $f:\Gamma\to G$ be a map of finite groups  that  induces  surjection $f_{ab}\colon \gC_{ab}\twoheadrightarrow G_{ab}$ on  the abelianizations. Then
$U_\infty=U_N$ for some natural $N> 0,$ and it is a universal hypercentral $f$-extension  of $G:$
This extension maps uniquely to any other hypercentral $f$-extension of $G.$ In addition it  satisfies
$H_i(U_\infty,\Gamma)=0$ for $i=1,2$ and  it is terminal among all such factorizations under $f$ namely $f$-extensions $A\to E\to G$
with  $H_i(E,\Gamma)\cong 0$ for $i=1,2.$
\end{thm}

\section{$f$-central extensions and relative Schur multiplier}\label{sub f-central}
Recall that throughout this note
\[
f\colon\gC\to G
\]
is a {\it fixed map}.
The following definitions are direct generalizations of the usual definitions of
a group structure on extensions as in \cite{M},
the usual proofs extend here with only minor changes.

\begin{Def}\label{def ext}
A map between to $f$-extensions $M$ and $M'$ is a commutative diagram as below; these extensions
are called {\it equivalent} if $\kappa$ below is the
identity-and thus $\gt$  is an isomorphism:
\begin{equation}\label{diag map f-ext}
\xymatrix{
0\ar[r] &K\ar[r]\ar[dd]_{\kappa} & M\ar[rr]^n\ar[dd]_{\tau} &&G\ar[r]\ar[dd]^{=} & 1\\
&&&\ar[ru]^f\Gamma\ar[rd]^f\ar[lu]_{\psi}\ar[ld]_{\psi'}\\
0\ar[r]&K'\ar[r]&M'\ar[rr]^{n'} && G\ar[r]&1\\
}
\end{equation}

We denote by ${\rm Ext}^{\gC}(G,A)$ the set of all equivalence classes
of central $f$-extensions of $G$ with kernel $A$ (where the map $f$ is suppressed from the notation).
We now recall and extend the usual group
structure on ${\rm Ext}^{\gC}(G,A)$.  The additional information
that we insert is the map from $\gC$ to the added extension.
Let
\[
0\to A\xrightarrow{\ga} M\xrightarrow{\gb} G\to 1,\qquad\psi
\]
and
\[
0\to A\xrightarrow{\ga'} M'\xrightarrow{\gb'} G\to 1,\qquad\psi'
\]
be two $f$-extensions.  Consider the following diagram
\[
\xymatrix{
0\ar[r] &A\times A\ar[r]^{\ga\times\ga'} &M\times M'\ar[r]^{\gb\times\gb'} &G\times G\ar[r] &1\\
0\ar[r] &A\times A\ar[d]_+\ar[u]^{{\rm id}\times{\rm id}}\ar[r]^{\gc} &PB\ar[d]\ar[u] \ar[r] &G\ar[u]_{\gD}\ar[r]\ar[d]^= &1\\
0\ar[r] &A\ar[r]^{\iota} &PO\ar[r] &G\ar[r] &1
}
\]
Here, $\gD(g)=(g,g)$ is the diagonal map, $PB$ is the pullback  in the right upper square,
the map $A\times A\xrightarrow{+} A$ is addition, and $PO$ is the push-out in the left lower square.
That push-out maps to $G$ by the universal property of push-outs.
Notice that we have the map $\psi\times\psi'\colon \gC\to M\times M'$ and the map $f\colon\gC\to G,$
and that $(\psi\circ\psi')\circ(\gb\times\gb')=f\circ\gD,$ so by the universal property
of pullbacks there is a map $\gC\to PB,$ and consequently a map $\psi''\colon \gC\to PO$.
Thus $(PO,\psi'')$ is the sum of the given extensions. (Notice that since $\ga\times\ga'$
is injective, so is $\gc$ and consequently also $\iota$).

The inverse of an element is its  push-out along $-1: A\to A.$
\end{Def}

\begin{examples}\label{eg equivalence}
\begin{enumerate}
\item
Assume that $\gC=G$ and that $f$ is the identity map.
Then any two $f$-extensions of $G$ with kernel $A$
are isomorphic.  Indeed, if $(M,\psi)$ and $(M',\psi')$ are such
$f$-extensions,
then we may assume that $A$ is contained in both $M$ and $M',$
and then $M=A \rtimes\psi(G),$ the semi direct product, and $M'=A\rtimes\psi'(G)$
and the map $1_A\times(\psi^{-1}\circ\psi')$ is an isomorphism
of extensions.

\item
Assume that $G=\{1\}$.  Then one has $M\cong A$ and there is only one
equivalence class of central extensions of $G$ with kernel $A$.  However,
for any map $\psi\colon\gC\to A$ we get that $(M,\psi)$ is an $f$-extension
 of $G$, and note that distinct maps $\psi$ will give inequivalent
central $f$-extensions hence ${\rm Ext}^{\gC}(1;A)\cong Hom(\Gamma,A).$

\item
The neutral element of ${\rm Ext}^{\gC}(G;A)$ is
a (semi-direct product) split extension\linebreak $0\to A\to A\rtimes G\to G\to 1,$
coming from the action of $G$ on $A,$ where $\psi_0\colon\gC\to A\rtimes G$
is the map $\psi_0(\gc)=(0,f(\gc)),$ for all $\gc\in\gC$.

\item
Note that if $f$ itself is a central extension of $G$, then it is itself a (universal)  central $f$-extension
as in the theorem above,
with the identity  on $\Gamma$ as the structure map.
\end{enumerate}
\end{examples}

Notice that

\begin{lemma}\label{lem pullback}
Let $M$ be an extension of $G$ with abelian  kernel $A,$ and
let $M'$ be the pull-back of $M$  along $f$ as in diagram \eqref{diag pullback} below:
\begin{equation}\label{diag pullback}
\xymatrix{
M'\ar[r]^{n'}\ar[d]_{\pi} & \gC\ar[d]^f\\
M\ar[r]^{n} & G
}
\end{equation}
then
\begin{enumerate}
\item
if $M$ is central then $M'$ is a central extension of $\gC$ with the same kernel $A$;

\item
if $(M,\psi)$ is an $f$-extension, then $M'$ splits;

\item
conversely, if $n'$
has a section $s\colon\gC\to M',$  then $(M,\psi)$ is an $f$ -extension, where $\psi=s\circ\pi\colon \gC\to M$.
\end{enumerate}
 \end{lemma}
\begin{proof}
Recall that $M'=\{(m,\gc)\mid m\in M,\ \gc\in\gC\text{ and } n(m)=f(\gc)\}$.
Also $n'$ and $\pi$ are the corresponding projection maps.
Hence part (1) is clear.
\medskip

\noindent
(2):\quad
Consider the map $s\colon\gC\to M'$ defined by $s(\gc)=(\psi(\gc),\gc),$ for all $\gc\in\gC$.
Then $s\circ n'=1_{\gC},$ is the identity on $\gC,$ so $s$ is a section of $n'$.
\medskip

\noindent
(3): Clearly
$\psi=s\circ\pi$ shows that $(M,\psi)$ is an $f$-extension of $G$.
\end{proof}

\subsection*{2.4.\ Definition of relative homology and cohomology $H_*(G,\gC;A)$ and $H^*(G,\gC;A)$}\hfill
\medskip

\noindent
Given a map $f\colon\gC\to G,$ and an abelian group $A,$ we now consider  the
relative (co)homology of $f$ with coefficients in $A$. The obvious candidate obtained by taking classifying spaces
works well, see below, but here is an algebraic  exposition  that follows the  usual chain complex approach.

The (co)homology group is defined as the (co)homology of the (co)chain complex
gotten as the {\it mapping cone} (compare \cite[1.5, p.18]{We}).
One starts with the usual chain map induced by the group map $\Gamma \to G$:
\[
C_*(f)\colon C_*(\Gamma;\zz)\to C_*(G;\zz).
\]
Define $C_*(G,\Gamma;\zz)$ to be the mapping cone of $C_*(f)$:
%
\begin{equation}\label{eq relative chain}
C_n(G,\Gamma):= C_n(G)\oplus C_{n-1}(\Gamma)\quad\text{ with }\quad
d(x,y) =(dx-C_{n-1}(f)(y),dy ) \in C_{n-1}(G,\Gamma).
\end{equation}
%
Now proceed as usual-noting that the chains above are dimension-wise free abelian groups- to define the (co)chain complexes with coefficients in $A$ to be:
\begin{equation}\label{eq coefficients}
C_*(G,\Gamma;A):=C_*(G,\Gamma)\otimes A,\qquad C^*(G,\Gamma;A):={\rm Hom}(C_*(G,\Gamma), A).
\end{equation}
%
This completes the definition for coefficients in any abelian group with a trivial $G$ action.

\medskip
\noindent
{\bf Remark}: If one considers  a general module  $M$ over $G,$  it  induces  a module over $\Gamma$ and one must proceed more
carefully to define the appropriate relative chains and co-chains. We shall not pursue it here. Thus from now on we consider only
central extensions and (co-)homology with abelian coefficients.
\smallskip

The map $C^*(f)$ together with the mapping cone gives a long exact sequences in (co)homology.
In addition, from equation \eqref{eq coefficients}, one concludes that the
relative homology and cohomology  are related by the usual {\it universal coefficients theorem}
as in \cite{Sp, We}. (Of course, there is a topological
analogue, see below.)

Looking at the boundary formula of equation \eqref{eq relative chain},
ones sees that a co-cycle element in $C^n(G,\gC;A)$ is a pair $(a,b),$ where $a\in C^n(G;A)$ is a co-cycle,
and $b\in C^{n-1}(\gC,A)$ is an explicit null cohomology of the pullback of $a$ to $C^{n}(\gC;A)$.

By construction there is a sequence of three chain complexes
\[
C^*(G,\gC;A)\to C^*(G;A)\to C^*(\gC;A).
\]

This sequence gives, as usual for mapping cones, a {\it long exact sequence} relating the
usual cohomology groups to the relative cohomology groups.

\subsubsection*{2.5.\ A topological definition}\hfill

\noindent
The group cohomology of $G$ is in fact the cohomology  of the
associated classifying space $BG,$  via  the corresponding  chain complexes $C_*(G)$.
The algebraic and  topological complexes  are equivalent as chain complexes,
hence
we can take the above relative cohomology
as the relative  cohomology, namely
$H^*(G,\Gamma; A)=H^*(BG, B\gC;A)$
and similarly for homology with coefficients in an abelian group $A.$

Since  coefficients {\em are constant,} i.e.~the trivial module over the groups involved,
a relative cohomology  class  corresponds to a homotopy
class of pointed maps $[u]:BG\to K(A,2),$ together with a class of null-homotopy $[\nu]$
of the pre-composition of $[u]$ to $B\Gamma$.  This $\nu$ gives the extension of $u$ to the cone.
Further $\nu$  corresponds exactly
to a lift of  $u\circ B(f)$ to the homotopy fibre of $u$.  This lift is a map
$B\Gamma \to BE,$  were $E$ is the central extension of $G$ corresponding,
as usual, to the algebraic cohomology class corresponding to the class $[u]$.
A similar statement is true for local coefficients except that we need the usual twisted classifying space
to represent a cohomology class.
Proposition \ref{prop f-central}  below gives a detailed algebraic formulation
of the topological observation above.
\medskip

We now use the above basic definitions and observations in the following result, which, in
details and spirit, is a direct extension of the classical classification result for central extensions of group
via the second cohomology. An analogous result is true for general extensions with abelian kernel and both proofs are versions of
 the usual proof  for the non-relative case as in \cite{M}.

\setcounter{prop}{5}
\begin{prop}\label{prop f-central}
The abelian group  ${\rm Ext}^{\gC}(G;A)$  of  equivalence classes of central $f$-extension of $G$ with a given abelian group $A$ as  the central  kernel, are classified by
the relative cohomology group $H^2(G,\gC; A)$ as above.
\end{prop}
\begin{proof}
Namely we construct a natural isomorphism:
\[
{\rm Ext}^{\gC}(G;A)\cong H^2(G,\gC; A)
\]
The proof is a simple extension of the classification of usual central extensions of
$G$ via  the second cohomology group $H^2(G,A)$.
Let ${\rm Ext}^{\gC}(G;A)$ be
as defined in Definition \ref{def ext}.

The central extensions of $G$ with kernel  $A$ are classified, up to equivalence, by the second cohomology
$H^2(G;A)$.  Now the restriction map  $f^*\colon H^2(G; A)\to H^2(\Gamma;A)$ takes a central extension
in $H^2(G;A)$ to its pull back along $f$.  By Lemma \ref{lem pullback},
the extensions of $G$ that are $f$-central  are precisely the extensions that are sent to the trivial extension in
$H^2(\Gamma;A)$. This gives an  exact sequence
\[
{\rm Ext}^\Gamma(G,A)\to H^2(G,A)\xrightarrow{f^*} H^2(\Gamma,A).
\]
We show that this sequence can be extended to a $5$-term exact sequence which is the same as the
usual exact sequence for $H^2(G,\Gamma; A),$ namely the sequence:
\begin{equation}\label{eq exact H2}
H^1(G;A)\to H^1(\gC; A) \to H^2(G,\gC; A)\to H^2(G,A)\xrightarrow{f^*} H^2(\gC,A),
\end{equation}
but with $H^2(G,\gC;A)$ replaced by ${\rm Ext}^{\gC}(G;A)$.

We first define a map from $H^1(\gC,A)\cong {\rm Hom}(\gC,A)$
onto the {\it kernel} of the map ${\rm Ext}^\Gamma(G,A)\to H^2(G,A)$.
Given $\mu\colon\gC\to A$ we define an element in ${\rm Ext}^{\gC}(G;A)$
as follows.  We let $M=A\times G,$ with the obvious maps $A\to A\times G\to G,$ and we let $\psi\colon\gC\to M$ be defined
by $\psi(\gc)=(\mu(\gc), f(\gc)),$ for all $\gc\in\gC$.  Then $(M,\psi)\in {\rm Ext}^{\gC}(G;A)$
is in the {\it kernel}.  Conversely given
$(M,\psi)$
in the {\it kernel}, it is a split extension $0\to A\xrightarrow{\ga} M\to G\to 1,$ with section $s\colon G\to M$.
Thus $M=\ga(A)\times s(G)$ and the map $\mu=\psi\circ\pi\circ\ga^{-1}\colon\gC\to A$  defines a central $f$-extension
of $G$ equivalent to $(M,\psi),$ where $\pi\colon M\to \ga(A)$ is the projection.

Finally, by Example \ref{eg equivalence}(3), the composition map ${\rm Hom}(G,A)\to {\rm Hom}(\gC,A)$
 takes ${\rm Hom}(G,A)$ by the above construction to  elements in the kernel  of $H^1(\gC,A)\to {\rm Ext}^\Gamma(G,A)$
 by taking $\phi:G\to A$ to be the appropriate self map  $((a,g)\mapsto (a-\phi(g),g)$ of the product $A\times G.$
 \medskip

To conclude the proof of the proposition we define an isomorphism between the exact sequence
\eqref{eq exact H2} and the exact sequence:
\begin{equation}\label{eq exact Ext}
H^1(G;A)\to H^1(\gC; A) \to {\rm Ext}^{\gC}(G; A)\to H^2(G,A)\xrightarrow{f^*} H^2(\gC,A),
\end{equation}
%

To do that we define a natural map:
\[
{\rm Ext}^{\gC}(G; A)\to H^2(G,\Gamma;A),
\]
The five lemma  then implies that it is an isomorphism.

Given an $f$-extension $(E,\psi)$ we must assign to it an element  $(a,b)\in C^2(G,\Gamma,A)=C^1(\Gamma;A)\oplus C^2(G;A)$.
Let $[c]\in H^2(G;A),$ with $c \in C^2(G;A)$ a co-cycle, be the element corresponding to the  extension $E,$
obtained by ignoring the  map $f$. Then $c$ is built as usual by choosing a set-theoretical section  $s\colon G\to E.$

 Consider now the pullback $E'$ of $(E,\psi)$ along $f$. As we saw in Lemma \ref{lem pullback},
 $E'$ is a split extension of $\Gamma $ by $A,$ with an explicit group theoretical section which we
 denoted $\psi'$.  Further, since $E'$ is also the {\it set theoretic}
 pullback along $f,$ the identity map $\gC\to\gC$ together with the map $f\circ s$
 yields a  {\it set theoretic  section}
 $s'\colon \Gamma\to E'$. The section $s'$ gives the co-cycle $C^2(f)(c)$.
 Of course the co-cycle given by $s'$ is cohomologous to the co-cycle given
 by $\psi'$.  This gives a
 1-cochain
 $w\in C^1(\Gamma,A)$ with $\partial w=C^2(f)(c),$ where $C^2(f)\colon C^2(G;A)\to C^2(\gC;A)$ is the induced
 map on chains.

We have assigned the pair $(w,c)\in C^2(G,\Gamma;A)$ to the
$f$-central extension $E.$ This construction is seen, as usual, to give a co-cycle in $C^2(G,\Gamma;A)$
whose cohomology class is independent of the choice of the set theoretic section $s$.
Moreover, by naturality, this definition of the map ${\rm Ext}^{\gC}(G;A)\to H^2(G,\gC;A)$
is compatible with  the maps between the $5$ terms exact sequences \eqref{eq exact Ext} and \eqref{eq exact H2},
where the other vertical maps are identity maps.
\end{proof}

\begin{remark}
A proof similar to the proof of Proposition \ref{prop f-central}
shows that the proposition holds with central $f$-extensions of $G$
replaced by $f$-extensions of $G$ with a given $G$-module $A$.
\end{remark}

\begin{cor}\label{cor surjection}
Assume that $f$ induces a surjection $f_{ab}\colon \gC_{ab}\to G_{ab},$ and let
$A$ be an abelian group. The group of central $f$-extensions with kernel $A:$
 ${\rm Ext}^{\gC}(G;A)$ is naturally isomorphic to ${\rm Hom}(H_2(G,\gC), A)$.
\end{cor}
\begin{proof}
As we noted in subsection 2.4, the universal coefficient theorem applies
in our set-up to yield the exact sequence:
\[
0\to {\rm Ext}(H_1(G,\gC);A)\to H^2(G,\gC;A)\to {\rm Hom}(H_2(G,\gC),A)\to 0.
\]
But the short exact sequence on 1-dimensional  homology:
$H_1(\Gamma)\to H_1(G)\to H_1(G,\Gamma)\to 0,$ and our surjection assumption,  yields the  vanishing of the  first
relative homology group $H_1(G,\Gamma),$  and thus  the isomorphism.
\end{proof}

The  construction in Definition \ref{def universal} below,
extends the   Schur  universal
central extension of a perfect group $G$ (the case $\gC=1$), to any map between groups
(not necessarily perfect groups) $f\colon \gC\to G,$ inducing surjection on abelianizations.

\begin{Def}\label{def universal}
Assume that $f$ induces a surjection $f_{ab}\colon \gC_{ab}\to G_{ab}$.
We define the {\it universal central $f$-extension of $G$}
as a pair $(U,\eta)\in {\rm Ext}^{\gC}(G; H_2(G,\gC))$ corresponding to the
identity map in ${\rm Hom}(H_2(G,\gC), H_2(G,\gC))$ (see Corollary \ref{cor surjection}).
\end{Def}

\noindent
{\bf Examples:}
\ Let the map $f$ be the abelianization $ f:\gC\to \gC_{ab}.$
Then the universal central $f$-extension is the following with the canonical structure map $\gC\to \gC/\gamma_3\gC$:
\[
\gamma_2\gC/\gamma_3\gC\to \gC/\gamma_3\gC\to \gC_{ab}.
\]
Of course, the higher quotients give analogous examples.

If the map $f$ is a central extension then the  associated universal central $f$-extension is $f$ itself with the identity as the structure map,
in this case $H_2(G,\gC)$ is canonically isomorphic  to the kernel of $f.$
Thus  for a quotient map of abelian groups $f: A\twoheadrightarrow B\cong A/K, $ we recover this quotient map as the universal  central $f$-extension.
Given a general surjective map $f:\gC\twoheadrightarrow G,$ the universal central $f$-extension is given by $K/[K,\gC]\to \gC/[K,\gC]\to G,$ with the natural quotient as the structure map. Of course if $\gC=1$ the trivial group we recover the Schur universal central extension of the perfect group $G.$

\medskip
When compared with the long exact homology sequence for a pair $G, \Gamma,$ it follows  that for any surjection of groups $\Gamma\to G$ with kernel $K$
we have naturally $H_2(G,\Gamma)\cong K/[K,\Gamma].$   The usual five terms   exact sequence for a group extension follows. When applied to a generators-relations  presentation of a group $G\cong F/R$
as a quotient of free group (see below)  one gets as  is well known, the standard version of Hopf formula for $H_2(G)$ as the kernel of $K/[R,F]\to Ker( F_{ab}\to G_{ab})$
which is clearly isomorphic to $R\cap[F,F]/[R,F]$ by a simple diagram chase.

\medskip
To get more interesting examples one can take any group map $\Gamma\to P,$ where $P$ is perfect:

\subsection*{\bf 2.10\ A relation to extensions of perfect groups and homological obstruction to lifting solutions.}\label{obstruction}
In this section we consider relative extensions of a perfect group $P$. The universal $f$-extension allows one to define obstructions to the lifting of
solutions of equations over $P$ to solutions over the universal Schur extension. Our discussion refers to the diagram
\ref{lifting problem} below.

 Given a perfect group $P,$
any map $f:\Gamma\to P$ is surjective on the abelian quotients, and therefore we can freely apply the results  as above.
 Take $\Gamma =C$ a cyclic group and consider a map $f_x: C\to P $ determined by  an element $x\in P.$
 Let $U_x$ be   the universal $f_x$-extension. We have $H_2(C)=0$ and
thus an exact sequence
\[\tag{$\star$}
0\to H_2(P)\to H_2(P,C) \to C \to 0.
\]
Let $E$ be the universal Schur extension of $P$ with kernel $H_2(P).$ By universality of $E$ one has a map of central extension  $e:E\to U_x$  (ignoring $f_x$).

\medskip

Let us consider the possibility of lifting the map $f_x:C\to P$ to $E.$ For $C=C_n,$ a cyclic group of order  $n,$ this would mean lifting
 an element of order $n$ in $P$ to an element of the same  order in $E$ which, in general, is impossible.

 \begin{equation}\label{lifting problem}
\xymatrix{
0\ar[r] &H_2(P)\ar[r]\ar@{^{(}->}[dd]_{1-1} & E\ar[rr]\ar@{^{(}->}[dd]^e &&P\ar[r]\ar[dd]^{=} & 1\\
&&&\ar[ru]^{f_x}C\ar[rd]_{f_x}\ar[lu]_{\tilde f_x}\ar[ld]\\
0\ar[r]&H_2(P,C)\ar@{->>}[d]\ar[r]&U_x\ar[rr]\ar@/^10pt/@{->>}[uu]^{u_x}\ar@{->>}[d] && P\ar[r]&1\\
&C\ar[r]^=&C\\}
\end{equation}
 Assume that $\tilde f_x$ is a lift of $f_x$ to $E.$  Then both extensions are equipped with a map from $\Gamma=C$ over $P,$ and  thus are $f_x$-extensions. The universality of $U_x$ now gives a map of $f_x$-extensions  $u_x:U_x\to E$ left inverse to $e.$
If $C=\zz,$  the integers, a lift $\tilde f_x$ as assumed always exists. In this case  the homology  sequence above  splits and
the { kernel of the universal   central $f_x$-extension  $U_x$ is isomorphic to $H_2(P)\oplus \zz.$ } Since $U_x$ is by itself an extension
of the Schur extension  $E$ of $P,$ it must split by the  universal properties of $U_x$ and $E$  hence  the
universal extension splits as:   $U_x\cong E\oplus \zz.$  The structure map $\zz\to E\oplus \zz$ corresponds to the choice $\tilde f_x$ of the lift of $x\in G$ to  $E.$

When $C$ is a {\em finite} cyclic group, the sequence $(\star)$ above does not split in general but under our assumption
 of the existence of a lift  $\tilde f_x$  we conclude  that $U_x \to E$  has a section  and thus $(\star)$ does split.
 We conclude that {\em the splitting of $(\star)$ is a necessary condition for the existence of a  lift $f_x.$ } In other words the
 extension sequence $(\star)$ is an obstruction element in $Ext(C, H_2(P))$ to the lifting problem of $x\in P.$

Similarly, for any map $f:\Gamma\to P,$ and  let $U$ be the universal central $f$-extension. Assume that $f$  lifts to $\tilde f: \Gamma\to E,$
 then,  one must have a splitting of $U\to E$ as argued above,
 since both  $E$ and $U$ have their  initial-universal properties.   Note that the canonical map $e:E\to U$  {\em is not} a map of central $f$-extension while $u:U\to E$ is such a map, thus we have
 $e\circ u=Id_E$ but in general $u\circ e$ is not the identity on $U.$

\medskip

Looking at the kernels of the extensions one gets that $H_2\Gamma\to H_2 P$ is the zero map.

\medskip

Consider for example a map $f:\Gamma=\zz\oplus \zz\to P $  given by a choice of
two commuting elements   $x,y\in P.$  This map, in general,  cannot be lifted to the Schur extension $E$ of $P.$
We look for a necessary condition for such a lift to exists, namely for an obstruction for lifting this pair of commuting elements
in $P$ to commuting elements in $E.$
Assume we have a lift   of $f$ i.e. $\tilde f: \zz\oplus\zz\to E.$

A lift gives a map $U\to E$ by universality of $U,$ thus getting a {\em split}  short exact sequence, since
$E\to U$ always exists by the universality of $E:$
\[
1\to \zz\oplus\zz\to U \to E\to 1.
\]
This means that the natural map $H_2(\Gamma)=H_2(\zz\oplus\zz)=\zz \to H_2(P)$ induced by the map $f$
 must be zero, as argued above. Hence the relevant element in the  second homology is an obstruction to the lift:  The injection $E\to U$ induces a injection on the kernels of the extensions: $H_2(P)\hookrightarrow H_2(P,\zz\oplus\zz).$

  More generally, starting with  elements $x_i \in G$ satisfying  equations $w_a(x_1,x_2...)=1$  one gets a homological
 obstruction for lifting $x_i$ to   elements  $e_i\in E$ satisfying the same conditions $w_a(e_i)=1$  taking $\Gamma=Free/w_a$
 namely the vanishing of $H_2(\Gamma)\to H_2 (P),$ where $H_2(\Gamma)$ can be given explicitly in terms of these relations.

\subsection*{\bf 2.11 The universality of U}\hfill
\medskip

\noindent
We now turn to  the universal $f$-extension $U$ --- a relative version of the universality of
the  Schur central extension of perfect groups.

\setcounter{prop}{11}
\begin{thm}\label{thm f-central}
Assume that the map $f_{ab}\colon \gC_{ab}\to G_{ab}$ induced on the abelianizations is
surjective, and let $(U,\eta)$ be the universal central $f$-extension of Definition \ref{def universal}.
Then for any central $f$-extension  $(E,\psi)$ of $G$
there is a unique map of central $f$-extensions (as in Definition \ref{def f-central}(2))
from $U$ to $E$. The map on the abelianization $\Gamma_{ab}\to U_{ab}$ is surjective.
\end{thm}
\begin{proof}
By Corollary \ref{cor surjection}, there is a $1-1$ correspondence between central $f$-extensions
of $G$ with kernel $A$ and ${\rm Hom}(H_2(G,\gC),A)$.  Furthermore

\begin{lemma}\label{lem universal}
The isomorphism of Corollary  \ref{cor surjection} can be built for the universal $f$-extension $(U,\eta)$ by  taking $e\in {\rm Hom}(H_2(G,\gC), A)$
to the element $(E,\psi)$ in ${\rm Ext}^{\gC}(G;A)$    as in the
following pushout diagram, where $E$ is the pushout of $\iota$ and $e,$
and $\psi=\eta\circ\gt$.
\begin{equation}\label{diag pushout}
\xymatrix{
0\ar[r] &H_2(G,\gC)\ar[r]^{\iota}\ar[dd]_e & U\ar[rrd]^n\ar[dd]_{\tau} \\
&&&\Gamma\ar[lu]^{\eta}\ar[ld]_{\psi}\ar[r]^<<f &G\ar[r] &1\\
0\ar[r]&A\ar[r]&E\ar[rru]_{m}\\
}
\end{equation}
\end{lemma}
\begin{proof}
By naturality, given any map $a\in{\rm Hom}(H_2(G,\gC), A),$ the following diagram commutes
\begin{equation}\label{diag nat push}
\xymatrix{
{\rm Ext}^{\gC}(G,H_2(G,\Gamma))\ar[r]^<<<<{\cong}\ar[d]^{push\,\,a}&  {\rm Hom}(H_2(G,\Gamma), H_2(G,\Gamma))\ar[d]^{a \circ\,\text{--}}\\
{\rm Ext}^{\gC}(G,A)\ar[r]^<<<<<<<<{\cong}&{\rm Hom}(H_2(G,\Gamma), A)\\
}
\end{equation}
where the map ``push $a$'' is the map obtained by a pushout diagram as in
diagram \eqref{diag pushout} above, with $(U,\eta)$ replaced by some element
in ${\rm Ext}^{\gC}(G,H_2(G,\Gamma)),$ and where $e$ is replaced by $a$.  The map $a\circ$\,-- is composition with $a$.

Let $(M,\psi)\in {\rm Ext}^{\gC}(G,A)$. It corresponds uniquely to $a\in {\rm Hom}(H_2(G,\Gamma), A)$.
Now consider diagram \eqref{diag nat push} with {\it that map}  $a$.  Then $a=a\circ{\rm id}_{H_2(G,\gC)},$ i.e, $a$ comes from the identity
in  ${\rm Hom}(H_2(G,\Gamma), H_2(G,\Gamma)),$ and we see that $(M,\psi)$ occurs as a pushout
of $(U,\eta)$.
\end{proof}

Lemma \ref{lem universal}
shows that each $f$-central
extension $(E,\psi)$ of $G$ with kernel $A$ occurs uniquely as a pushout
of $e\in {\rm Hom}(H_2(G,\gC),A)$.  The desired map $U\to E$ is the map $\gt$ of diagram \eqref{diag pushout}.
\end{proof}

\begin{remark}
The results of this section were used in \cite[Theorem 5.4]{FS} to identify
the kernel of the free normal closure $\gC^f,$ in the case where $G=\lan f(\gC)^G\ran$.
\end{remark}

In the above paragraphs we have proved the main part of theorem \ref{thm f-central}.
To conclude the proof of Theorem \ref{thm f-central}  we consider a basic property of  the universal extension  that extends a familiar
property of the universal central extension of a perfect group:

\begin{lemma}\label{lem surj on ab}
If $f$ induces a surjection on abelianization then the same is true for the universal extension $(U,\eta)$:  the  associated structure  map
of the extension
$\eta: \Gamma\to U$ is also surjective on the abelian quotients.
\end{lemma}
\begin{proof}
Here we provide two proofs, one direct and one using  a 5-term exact sequence.

For a direct proof, first note that
$\eta_{ab}$ is surjective if and only if 
$\eta\mapsto \eta\circ\mu$ is an injective map  ${\rm Hom}(U,A)\to{\rm Hom}(\gC,A),$
for any abelian group $A$.
Now given a map $a: U\to A$ to any abelian group  whose pre-composition to $\Gamma$  is zero,
we  show that $a=0:$ Consider the trivial $f$-extension
of $G$ namely $A\times G;$
where this trivial extension is equipped with the  structure  map  $0\times f:\Gamma\to A\times G.$

Let $u\colon U\to G$ be the unique map of Theorem \ref{thm f-central}.  Now   construct  two maps  between these $f$-extensions $(U,\eta)$ and $(A\times G,0\times \eta),$ namely
the maps $a\times u , 0\times u: U\to A\times G.$

The condition on the pre-composition
  implies that these two maps {\em are maps of $f$-extensions} as defined above,
    since both are $0\times f$  on $\Gamma,$  commutativity  on $\Gamma$ follows.

   Thus both are well defined
maps of factorizations and therefore by uniqueness of maps  of factorization from the universal $(U,\eta),$  they are equal and we get $a=0$ as needed.
\medskip

Here is  a second proof using a {\em relative 5-terms exact sequence.}    Start with any central $f$-extension of the form $A \to E\to G$ with $\psi: \Gamma \to E$ as the structure map.  We have the  exactness of the following relative version of the usual
5-term  sequence.
\[\tag{$\star\,\star$}
H_2(E,\Gamma)\to H_2(G,\Gamma)\overset \partial\rightarrow A\to H_1(E,\Gamma)\to H_1(G,\Gamma)\to 0.
\]
To see this one can proceed algebraically  by chasing the relevant diagram of exact sequences, but here is a topological argument:

Comparing the exact sequence of the cofibration $BE\to BG\to BG/BE$ with
the usual 5-term exact sequence of lower homology groups  (i.e. with $\Gamma=1$ in $(\star\,\star)$   above)  implies (by the 5-lemma)
the isomorphism $ H_2(BG/BE) \cong A.$

Therefore the  homology exact sequence, in low dimensions, of the cofibration sequence involving three mapping cones:
\[
 BE/B\Gamma \to BG/B\Gamma \to BG/BE,
\]
yields the desired 5-terms  exact sequence in homology.

Now  consider the  5-terms exact sequence  in relative homology associated to the universal  extension sequence:
\[
0\to H_2(G,\Gamma)=A\to U\to G\to 1
\]
of the universal central $f$-extension
with the structure map $\Gamma\to U$:

We note that the  connecting boundary   map $\partial$  in the low homology sequence is an isomorphism by construction. We have by assumption
 that $H_1(G,\Gamma)\cong 0,$
we get immediately for the exact sequence the vanishing  $H_1(U,\Gamma)\cong 0,$ and thus surjectivity on the first homology as needed.
\end{proof}
\section{Universal factorizations}

Next we consider the universal properties   of $U_\infty,$ as stated in   theorem \ref{thm uni U}  of the introduction:
 Characterizing the factorization $\Gamma \to U_\infty\to G$ as both initial and terminal.

\subsection*{3.1\ Relative Schur Tower} In light of Lemma \ref{lem surj on ab},
we can repeat the construction of the universal central $f$-extension.
We get a tower
\[
 \Gamma \to U_\infty\to \cdots U_{n+1}\to U_n\to   \cdots   U_1=U\to U_0=G.
\]
As we will presently see,
for finite groups,   it stop  after a finite number of steps at {\em the universal  hypercentral extension of $G.$}

\setcounter{prop}{1}
\begin{prop}
The above tower stops: For any map $f\colon\gC\to G$ of finite groups as above, with $f_{ab}$ surjective,
the inverse limit $U_\infty$ is a finite group which is
equal to $U_n$ for all sufficiently large $n.$
\end{prop}
To prove that the tower stops we use  a slight generalization of the argument given in Lemma 4.4 and Lemma 4.5 of \cite{FS}.
The same proof goes through using the following slight modification which  is a basic observation on nilpotent groups

\begin{lemma}
Let $g: G\to  N$ be a map  with $N$ being nilpotent, such that $g$ induces a surjection  on abelianization
then $g$ is surjective. In particular, if $g$ is also injective then $G=N.$
\end{lemma}
\begin{proof}
The only property of the  nilpotent group $N$ used here is that its Frattini quotient $N/\Phi(N)$ is an abelian group.
We can assume $G$ is a subgroup of $N$ by looking at its image under $f$. The assumption of surjection $G \twoheadrightarrow N_{ab}$
implies that $G$ together with the commutator subgroup of $N$ generate $N,$ i.e. $N=[N,N]G$.
Since $[N,N]\le\Phi(N),$ we get    $G\Phi(N)=N.$ But by the basic property of $\Phi(N)$ this means $G=N.$\qedhere
\end{proof}

\noindent
{\bf Remark:}
Topologically the  relative Schur tower  is gotten by repeatedly taking the  homotopy  fibre of the composition:
\[
BG\to BG/B\Gamma\to P_2( BG/B\Gamma )
\]
as a sort of ``relative modified Bousfiled-Kan homology completion tower".
Moreover, this last construction is well defined without any assumption of $\Gamma\to G.$

\begin{prop}
Let $f\colon\gC\to G$ be a map of finite groups, surjective on abelian quotients.
The limit  $U_\infty$  of the tower $U_i$ of repeated  universal central $f$-extensions is
{\em terminal} among all   factorizations of $\Gamma \to E  \twoheadrightarrow G$  of  $f$
which are low-dimensional-acyclic factorization namely with $H_i(E,\Gamma)=0$ for $i=1,2.$
\end{prop}
\smallskip

\noindent
The  initial extension here  is the trivial $\Gamma\to M=\Gamma \to G$.

Dually one has
\begin{prop}
Let $f\colon\gC\to G$ be a map of finite groups, surjective on abelian quotients.
The limit  $U_\infty$  of the tower $U_i$ of repeated  universal central $f$-extensions is {\em initial} among all  hypercentral 
factorization
$f:\Gamma \to M  \twoheadrightarrow  G$  of  maps $f:\Gamma \to G$. In addition it is lower acyclic namely,
$H_i(U_\infty,\Gamma)=0$ for $i=1,2.$
\end{prop}

The terminal hypercentral extension is of course $\Gamma\to M=G\to G.$
\medskip

\noindent

\begin{proof}[Proofs of universality]
To prove the initial property of $U_\infty$ we use of course that $U=U_1$ itself is initial
among all central $f$-extensions. Now we proceed by induction. Given any hypercentral extension  of $G$ it comes with a finite tower
$E_n\to E_{n-1}\to\cdots\to  E_1 \to G,$ where  $E_{i+1}\to E_i$ etc.  is a central $f$-extension. Now we assume we have a map $U_\infty=U_N$ to $E_i.$
then by the universality of $U$ we also have a map $U_{N+1}\to  E_{i+1}$. But we know that $U_{N+1}=U_N=U_\infty.$ Hence  we get a map of
hypercentral $f$-extensions. Uniqueness follows similarly, compare with \cite{FS}.

We turn to the terminal property of $U_\infty.$  Let $\Gamma\to M\to G$ be a factorization by some $f$-extension with $H_i(M,\Gamma)=0$ for
$i=1,2$.  Consider the universal factorization  $\Gamma\to W_\infty \to M$ of $\Gamma\to M.$ It maps to $U_\infty.$ The map $M\to G$ induces by naturality a map
$W_\infty\to U.$ But we saw above that the kernel of $W\to M$ is $H_2(M,\gamma)= 0$ by assumption on $M.$ Therefore  $W=M$ we get a map $M\to U.$
repeating this we  get a map  $M\to U_\infty.$ Uniqueness follows by a similar argument using $H_2(U_\infty,\Gamma)=0.$
\end{proof}
\medskip

\noindent
{\bf A remark on a relation between $U_\infty$ and $\Gamma_\infty$:}

\medskip

Together with the $\Gamma_\infty$ factorization of \cite{FS} we have  two universal constructions for maps of finite groups: One,  $\Gamma\to \Gamma_\infty\to G, $  is {\bf initial} among all  subnormal factorizations  i.e. $\Gamma\to S\to G$  with $S\to G$ a {\em subnormal map;} the other,
 $\Gamma\to U_\infty\to G$ is  {\bf terminal } among all factorization with $H_{i\leq 2}(-,\Gamma)=0.$ Both have trivial relative
homology groups $H_{*\leq 2}(-,\Gamma)\cong 0.$ Note that in  the  second construction above, $U_\infty,$ it is assumed that the map $\Gamma\to G$ is surjective on the  abelian quotients.

However, one may well start with an arbitrary map $f:\Gamma\to G$  and proceeds    by first taking   the maximal normal  subgroups  $B\subseteq  G$ containing the image of $f$  for which the map $f_B:\Gamma \to B$ is surjective on the abelian quotient. such a maximal subgroup
exists for any map in a natural way. In this way the universal  $f_B$ hypercentral extension  $\Gamma\to U_B\to B$ gives   a subnormal map $U_B\to G$ factorizing $f$ in the desired way. Thus both $\Gamma_\infty $ and $U_\infty:=U_B$ (by definition for an arbitrary map) are  defined for arbitrary maps of finite groups and both
give  subnormal factorizations in the sense of \cite{FS}; and now again  $\Gamma_\infty$ is initial and $U_\infty$ as defined here is terminal.
Since $U_\infty$ is subnormal by construction being a hypercentral extension of a normal subgroup $B\unlhd G,$ there is a unique natural map
$\Gamma_\infty\to U_\infty $ commuting with the factorization diagrams. Dually this map exists because of the  initial property of $U_\infty$
 as above.  For example, if we take $G=P$ a perfect group and $\Gamma=1$ the trivial
group then $\Gamma_\infty=1$ and $U_\infty=E$ the universal Schur extension of $P.$


\end{document}